\documentclass[twoside]{article}

\usepackage[round]{natbib}
\usepackage{amsmath,amssymb,amsthm,natbib,pgf,caption,subcaption,placeins}
\usepackage{macros}
\newtheorem{theorem}{Theorem}

\newtheorem{assumption}[theorem]{Assumption}
\newtheorem{lemma}[theorem]{Lemma}

\usepackage{algorithm}
\usepackage[noend]{algpseudocode}

\usepackage[accepted]{aistats2016}

\begin{document}

%

%

\twocolumn[

\aistatstitle{A Linearly-Convergent Stochastic L-BFGS Algorithm}

\aistatsauthor{ Philipp Moritz \And Robert Nishihara \And Michael I. Jordan }

\aistatsaddress{ University of California, Berkeley \\ \texttt{pcmoritz@eecs.berkeley.edu} \And University of California, Berkeley \\ \texttt{rkn@eecs.berkeley.edu} \And University of California, Berkeley \\ \texttt{jordan@eecs.berkeley.edu}} ]

\begin{abstract}
We propose a new stochastic L-BFGS algorithm and prove a linear convergence rate for strongly convex and smooth functions. 
Our algorithm draws heavily from a recent stochastic variant of L-BFGS proposed in \citet{byrd2014stochastic} as well as a recent approach to variance reduction for stochastic gradient descent from \citet{johnson2013accelerating}. 
We demonstrate experimentally that our algorithm performs well on large-scale convex and non-convex optimization problems, exhibiting linear convergence and rapidly solving the optimization problems to high levels of precision. 
Furthermore, we show that our algorithm performs well for a wide-range of step sizes, often differing by several orders of magnitude. 
\end{abstract}

\section{Introduction}
A trend in machine learning has been toward using more parameters to model larger datasets. 
As a consequence, it is important to design optimization algorithms for these large-scale problems. 
A typical optimization problem arising in this setting is empirical risk minimization. 
That is,
\begin{equation} \label{eq:erm}
  \min_w \frac{1}{N}\sum_{i=1}^N f_i(w) ,
\end{equation}
where~$w \in \mathbb R^d$ may specify the parameters of a machine learning model, and~$f_i(w)$ quantifies how well the model~$w$ fits the~$i$th data point. 
Two challenges arise when attempting to solve \eqref{eq:erm}. 
First,~$d$ may be extremely large. 
Second,~$N$ may be extremely large. 

When~$d$ is small, Newton's method is often the algorithm of choice due to its rapid convergence (both in theory and in practice). 
However, Newton's method requires the computation and inversion of the Hessian matrix~$\nabla^2 f(w)$, which may be computationally too expensive in high dimensions. 
As a consequence, practitioners are often limited to using first-order methods which only compute gradients of the objective, requiring~$O(d)$ computation per iteration. 
The gradient method is the simplest example of a first-order method, but much work has been done to design quasi-Newton methods which incorporate information about the curvature of the objective without ever computing second derivatives. 
L-BFGS \citep{liu1989limited}, the limited-memory version of the classic BFGS algorithm, is one of the most successful algorithms in this space. 
Inexact Newton methods are another approach to using second order information for large-scale optimization. They approximately invert the Hessian in~$O(d)$ steps. 
This can be done by using a constant number of iterations of the conjugate gradient method \citep{dembo1982inexact,dembo1983truncated,nocedal2006numerical}. 

When~$N$ is large, batch algorithms such as the gradient method, which compute the gradient of the full objective at every iteration, are slowed down by the fact that they have to process every data point before updating the model. 
Stochastic optimization algorithms get around this problem by updating the model~$w$ after processing only a small subset of the data, allowing them to make much progress in the time that it takes the gradient method to make a single step. 

For many machine learning problems, where both~$d$ and~$N$ are large, stochastic gradient descent (SGD) and its variants are the most widely used algorithms \citep{robbins1951stochastic,bottou2010large,bottou2004large}, often because they are some of the few algorithms that can realistically be applied in this setting. 

Given this context, much research in optimization has been directed toward designing better stochastic first-order algorithms. 
For a partial list, see \citep{kingma2015adam,sutskever2013importance,duchi2011adaptive,shalev2013stochastic,johnson2013accelerating,roux2012stochastic,wang2013variance,nesterov2009primal,frostig2015competing,agarwal2014reliable}. 
In particular, much progress has gone toward designing stochastic variants of L-BFGS \citep{mokhtari2014global,wang2014stochastic,byrd2014stochastic,bordes2009sgd,schraudolph2007stochastic,sohl2014fast}. 

Unlike gradient descent, L-BFGS does not immediately lend itself to a stochastic version. 
The updates in the stochastic gradient method average together to produce a downhill direction in expectation. 
However, as pointed out in \citet{byrd2014stochastic}, the updates used in L-BFGS to construct the inverse Hessian approximation overwrite one another instead of averaging. 
Our algorithm addresses this problem in the same ways as \citet{byrd2014stochastic}, by computing Hessian vector products formed from larger minibatches. 

Though stochastic methods often make rapid progress early on, the variance of the estimates of the gradient slow their convergence near the optimum. 
To illustrate this phenomenon, even if SGD is initialized at the optimum, it will immediately move to a point with a worse objective value. 
For this reason, convergence guarantees typically require diminishing step sizes. 
One promising line of work involves speeding up the convergence of stochastic first-order methods by reducing the variance of the gradient estimates \citep{johnson2013accelerating,roux2012stochastic,defazio2014saga,shalev2013stochastic}. 

We introduce a stochastic variant of L-BFGS that incorporates the idea of variance reduction and has two desirable features. 
First, it obtains a guaranteed linear rate of convergence in the strongly-convex case. 
In particular, it does not require a diminishing step size in order to guarantee convergence (as partially evidenced by the fact that if our algorithm is initialized at the optimum it will stay there). 
Second, it performs very well on large-scale optimization problems, exhibiting a qualitatively linear rate of convergence in practice. 

\section{The Algorithm}

We consider the problem of minimizing the function
\begin{equation}
  f(w) = \frac{1}{N} \sum_{i=1}^N f_i(w) 
\end{equation}
over~$w \in \mathbb R^d$. 
For a subset~$\mathcal S \subseteq \{1,\ldots,N\}$, we define the subsampled function~$f_{\mathcal S}$ by
\begin{equation}
  f_{\mathcal S}(w) = \frac{1}{|\mathcal S|} \sum_{i \in \mathcal S} f_i(w)  .
\end{equation}

Our updates will use stochastic estimates of the gradient~$\nabla f_{\mathcal S}$ as well as stochastic approximations to the inverse Hessian~$\nabla^2 f_{\mathcal T}$. 
Following \citet{byrd2014stochastic}, we use distinct subsets~$\mathcal S,\mathcal T \subseteq \{1,\ldots,N\}$ in order to decouple the estimation of the gradient from the estimation of the Hessian. 
We let~$b=|\mathcal S|$ and~$b_H=|\mathcal T|$. 

Following \citet{johnson2013accelerating}, we occasionally compute full gradients, which we use to reduce the variance of our stochastic gradient estimates.

The update rule for our algorithm will take the form
\begin{equation*}
  w_{k+1} = w_k - \eta_k H_k v_k .
\end{equation*}
In the gradient method,~$H_k$ is the identity matrix. In Newton's method, it is the inverse Hessian~$(\nabla^2f(w_k))^{-1}$. 
In our algorithm, as in L-BFGS,~$H_k$ will be an approximation to the inverse Hessian. 
Instead of the usual stochastic estimate of the gradient,~$v_k$ will be a stochastic estimate of the gradient with reduced variance. 

Code for our algorithm is given in \algoref{alg:slbfgs}. 
Our algorithm is specified by several parameters. 
It requires a step size~$\eta$, a memory size~$M$, and positive integers~$m$ and~$L$. 
Every~$m$ iterations, the algorithm performs a full gradient computation, which it uses to reduce the variance of the stochastic gradient estimates. 
Every~$L$ iterations, the algorithm updates the inverse Hessian approximation. 
The vector~$s_r$ records the average direction in which the algorithm has made progress over the past~$2L$ iterations. 
The vector~$y_r$ is obtained by multiplying~$s_r$ by a stochastic estimate of the Hessian. 
Note that this differs from the usual L-BFGS algorithm, which produces~$y_r$ by taking the difference between successive gradients. 
We find that this approach works better in the stochastic setting. 
The inverse Hessian approximation~$H_r$ is defined from the pairs~$(s_j,y_j)$ for~$r-M+1 \le j \le r$ using the standard L-BFGS update rule, which is described in \secref{sec:construction_of_inverse_hessian}. 
The user must also choose batch sizes~$b$ and~$b_H$ from which to construct the stochastic gradient and stochastic Hessian estimates. 

\begin{algorithm*}
  \caption{Stochastic L-BFGS}
  \label{alg:slbfgs}
\begin{algorithmic}[1]
  \Require initial state~$w_0$, parameters~$m$,~$M$, and~$L$, batch sizes~$b$ and~$b_H$, and step size~$\eta$
  \State Initialize~$r=0$
  \State Initialize~$H_0=I$
  \For{$k = 0, \ldots$}
  \State Compute a full gradient~$\mu_k = \nabla f(w_k)$
  \State Set~$x_0=w_k$
  \For{$t = 0, \ldots, m-1$}
  \State Sample a minibatch~$\mathcal S_{k,t} \subseteq \{1,\ldots,N\}$
  \State Compute a stochastic gradient~$\nabla f_{\mathcal S_{k,t}}(x_t)$
  \State Compute a variance reduced gradient~$v_t=\nabla f_{\mathcal S_{k,t}}(x_t) - \nabla f_{\mathcal S_{k,t}}(w_k) + \mu_k$
  \State Set $x_{t+1} = x_t - \eta H_r v_t$
  \If{$t \equiv 0 \bmod L$}
  \State Increment $r \leftarrow r+1$
  \State Set~$u_r=\frac{1}{L} \sum_{j=t-L}^{t-1} x_j$
  \State Sample~$\mathcal T_r \subseteq \{1,\ldots,N\}$ to define the stochastic approximation $\nabla^2 f_{\mathcal T_r}(u_r)$
  \State Compute~$s_r = u_r-u_{r-1}$ 
  \State Compute~$y_r=\nabla^2 f_{\mathcal T_r}(u_r)s_r$
  \State Define~$H_r$ as in \secref{sec:construction_of_inverse_hessian}
  \EndIf
  \EndFor
  \State Set~$w_{k+1} = x_i$ for randomly chosen~$i \in \{0,\ldots,m-1\}$
  \EndFor
\end{algorithmic}
\end{algorithm*}

In \algoref{alg:slbfgs} and below, we use~$I$ to refer to the identity matrix. 
We use~$\mathcal F_{k,t}$ to denote the sigma algebra generated by the random variables introduced up to the time when the iteration counters~$k$ and~$t$ have the specified values. 
That is, 
\begin{equation*}
  \mathcal F_{k,t} = \sigma\left(
  \begin{array}{c} \{ \mathcal S_{k',t'} \,:\, k' < k \,\, \text{or} \,\, k' = k \,\, \text{and} \,\, t' < t \} \\
    \cup \, \{ \mathcal T_r : r L \le mk + t \} \end{array}
  \right) .
\end{equation*}
We will use~$\mathbb E_{k,t}$ to denote the conditional expectation with respect to~$\mathcal F_{k,t}$. 

We define the inverse Hessian approximation~$H_r$ in \secref{sec:construction_of_inverse_hessian}. 
Note that we do not actually construct the matrix~$H_r$ because doing so would require~$O(d^2)$ computation. 
In practice, we directly compute products of the form~$H_rv$ using the two-loop recursion \citep[Algorithm~7.4]{nocedal2006numerical}. 

\subsection{Construction of the Inverse Hessian Approximation~$H_r$} \label{sec:construction_of_inverse_hessian}

To define the inverse Hessian approximation~$H_r$ from the pairs~$(s_j,y_j)$, we follow the usual L-BFGS method. 
Let~$\rho_j=1/s_j^{\top}y_j$ and recursively define  
\begin{equation} \label{eq:inv_hess_update}
  H_r^{(j)} = (I - \rho_j s_j y_j^{\top})^{\top}H_r^{(j-1)}(I - \rho_j s_j y_j^{\top}) + \rho_j s_j s_j^{\top} ,
\end{equation}
for~$r-M+1 \le j \le r$. 
Initialize~$H_r^{(r-M)}=(s_r^{\top}y_r / \|y_r\|^2)I$ and set~$H_r=H_r^{(r)}$. 

Note that the update in \eqref{eq:inv_hess_update} preserves positive definiteness (note that~$\rho_j > 0$), which implies that~$H_r$ and each~$H_r^{(j)}$ will be positive definite, as will their inverses.

\section{Preliminaries}

Our analysis makes use of the following assumptions. 
\begin{assumption} \label{ass:diff}
  The function~$f_i \colon \mathbb R^n \to \mathbb R$ is convex and twice continuously differentiable for each~$1 \le i \le N$. 
\end{assumption}
\begin{assumption} \label{ass:hessian_bounds}
  There exist positive constants~$\lambda$ and~$\Lambda$ such that
  \begin{equation}
    \lambda I \preceq \nabla^2 f_{\mathcal T}(w) \preceq \Lambda I
  \end{equation}
  for all~$w \in \mathbb R^d$ and all nonempty subsets~$\mathcal T \subseteq\{1,\ldots,N\}$. 
Note the lower bound trivially holds in the regularized case.
\end{assumption}
We will typically force strong convexity to hold by adding a strongly-convex regularizer to our objective (which can be absorbed into the~$f_i$'s).  
These assumptions imply that~$f$ has a unique minimizer, which we denote by~$w_*$. 

\begin{lemma} \label{lem:trace_and_det_bounds}
  Suppose that \assref{ass:diff} and \assref{ass:hessian_bounds} hold. 
  Let~$B_r=H_r^{-1}$. 
  Then
  \begin{align*}
    \tr(B_r) & \le (d+M)\Lambda \\
    \det(B_r) & \ge \frac{\lambda^{d+M}}{((d+M)\Lambda)^M} .
  \end{align*}
\end{lemma}
We prove \lemref{lem:trace_and_det_bounds} in \secref{sec:proof_of_lem_trace_and_det_bounds}. 
\begin{lemma} \label{lem:hess_approx_bounds}
  Suppose that \assref{ass:diff} and \assref{ass:hessian_bounds} hold. 
Then there exist constants~$0 < \gamma \le \Gamma$ such that~$H_r$ satisfies
  \begin{equation}
    \gamma I \preceq H_r \preceq \Gamma I 
  \end{equation}
  for all~$r \ge 1$. 
\end{lemma}
In \secref{sec:proof_of_lem_hess_approx_bounds}, we prove \lemref{lem:hess_approx_bounds} with the values 
\begin{equation*}
  \gamma = \frac{1}{(d+M)\Lambda} \quad \text{and} \quad \Gamma = \frac{((d+M)\Lambda)^{d+M-1}}{\lambda^{d+M}} .
\end{equation*}

We will make use of \lemref{lem:strong_convex_bound}, a simple result for strongly convex functions. 
We include a proof for completeness. 
\begin{lemma} \label{lem:strong_convex_bound}
  Suppose that~$f$ is continuously differentiable and strongly convex with parameter~$\lambda$. 
  Let~$w_*$ be the unique minimizer of~$f$. 
  Then for any~$x \in \mathbb R^d$, we have
  \begin{equation*}
    \|\nabla f(x)\|^2 \ge 2\lambda(f(x)-f(w_*)) .
  \end{equation*}
\end{lemma}
\begin{proof}
By the strong convexity of~$f$, 
\begin{align*}
  f(w_*) & \ge f(x) + \nabla f(x)^{\top}(w_* - x) + \frac{\lambda}{2}\|w_*-x\|^2 \\
  & \ge f(x) + \min_v \left( \nabla f(x)^{\top}v + \frac{\lambda}{2}\|v\|^2\right) \\
  & = f(x) - \frac{1}{2\lambda}\|\nabla f(x)\|^2 .
\end{align*}
The last equality holds by plugging in the minimizer~$v=-\nabla f(x)/\lambda$.
\end{proof}

In \lemref{lem:bound_grad_var}, we bound the variance of our variance-reduced gradient estimates. 
The proof of \lemref{lem:bound_grad_var}, given in \secref{sec:proof_of_bound_grad_var}, closely follows that of \citet[Theorem~1]{johnson2013accelerating}.
\begin{lemma} \label{lem:bound_grad_var}
  Let~$w_*$ be the unique minimizer of~$f$. 
Let~$\mu_k=\nabla f(w_k)$ and let~$v_t=\nabla f_{\mathcal S}(x_t) - \nabla f_{\mathcal S}(w_k) + \mu_k$ be the variance-reduced stochastic gradient. 
  Conditioning on~$\mathcal F_{k,t}$ and taking an expectation with respect to~$\mathcal S$, we have
\begin{equation} 
    \mathbb E_{k,t}[\|v_t\|^2] \le 4\Lambda(f(x_{t})-f(w_*) + f(w_{k})-f(w_*)) .
\end{equation}
\end{lemma}

\section{Convergence Analysis}

\thref{th:linear_convergence} states our main result. 

\begin{theorem} \label{th:linear_convergence}
  Suppose that \assref{ass:diff} and \assref{ass:hessian_bounds} hold. 
  Let~$w_*$ be the unique minimizer of~$f$. 
  Then for all~$k \ge 0$, we have
  \begin{equation*}
    \mathbb E[f(w_{k})-f(w_*)] \le \alpha^k \mathbb E[f(w_0)-f(w_*)] ,
  \end{equation*}
  where the convergence rate~$\alpha$ is given by
  \begin{equation*}
    \alpha = \frac{1/(2m\eta) + \eta \Gamma^2 \Lambda^2}{\gamma \lambda - \eta\Gamma^2\Lambda^2} < 1 ,
  \end{equation*}
  assuming that we choose~$\eta< \gamma\lambda/(2\Gamma^2\Lambda^2)$ and that we choose~$m$ large enough to satisfy
  \begin{align}
    \gamma \lambda & > \frac{1}{2m\eta} + 2\eta\Gamma^2\Lambda^2 . \label{eq:eta_m_size}
  \end{align}
\end{theorem}
\begin{proof}
Using the Lipschitz continuity of~$\nabla f$, which follows from \assref{ass:hessian_bounds}, we have
\begin{align} \label{eq:more_bounds_in_proof} 
  & \,\, f(x_{t+1})  \\
  \le & \,\, f(x_{t}) + \nabla f(x_{t})^{\top}(x_{t+1} - x_{t}) + \frac{\Lambda}{2}\|x_{t+1} - x_{t}\|^2 \nonumber \\
  = & \,\, f(x_{t}) - \eta \nabla f(x_{t})^{\top} H_r v_{t} + \frac{\eta^2 \Lambda}{2} \|H_{k}v_{t} \|^2 . \nonumber
\end{align}

Conditioning on~$\mathcal F_{k,t}$ and taking expectations in \eqref{eq:more_bounds_in_proof}, this becomes
\begin{align} \label{eq:expansion_with_expectations}
  & \,\,\mathbb E_{k,t}[f(x_{t+1})]  \\
  \le & \,\, f(x_{t}) - \eta \nabla f(x_{t})^{\top} H_r \nabla f(x_{t}) + \frac{\eta^2 \Lambda}{2} \mathbb E_{k,t}\|H_{k}v_{t} \|^2 , \nonumber
\end{align}
where we used the fact that~$\mathbb E_{k,t}[v_{t}] = \nabla f(x_{t})$. 
We then use \lemref{lem:hess_approx_bounds} to bound the second and third terms on the bottom line of \eqref{eq:expansion_with_expectations} to get
\begin{equation*}
  \mathbb E_{k,t}[f(x_{t+1})] \le f(x_{t}) - \eta \gamma \|\nabla f(x_{t})\|^2 + \frac{\eta^2 \Gamma^2 \Lambda}{2} \mathbb E_{k,t}\|v_{t} \|^2 .
\end{equation*}
Now, we bound~$\mathbb E_{k,t}\|v_{t}\|^2$ using \lemref{lem:bound_grad_var} and we bound~$\|\nabla f(x_{t})\|^2$ using \lemref{lem:strong_convex_bound}. 
Doing so gives
\begin{align*}
  & \,\, \mathbb E_{k,t}[f(x_{t+1})] \\
  \le & \,\, f(x_{t}) - 2\eta \gamma \lambda (f(x_{t})-f(w_*)) \\
  & \quad + 2\eta^2 \Gamma^2 \Lambda^2(f(x_{t}) - f(w_*) + f(w_{k}) - f(w_*)) \\
  = & \,\, f(x_{t}) - 2\eta(\gamma \lambda - \eta\Gamma^2\Lambda^2) (f(x_{t})-f(w_*)) \\
  & \quad + 2\eta^2 \Gamma^2 \Lambda^2(f(w_{k}) - f(w_*)) .
\end{align*}
Taking expectations over all random variables, summing over~$t=0,\ldots,m-1$, and using a telescoping sum gives
\begin{equation*}
\begin{aligned}
  & \,\, \mathbb E[f(x_m)] \\
  \le & \,\, \mathbb E[f(x_0)] + 2m\eta^2 \Gamma^2 \Lambda^2 \mathbb E[f(w_{k}) - f(w_*)] \\
  & \quad - 2\eta(\gamma \lambda - \eta\Gamma^2\Lambda^2) \left(\sum_{t=0}^{m-1} \mathbb E [f(x_{t})]- mf(w_*)\right) \\
  = & \,\, \mathbb E[f(w_{k})] + 2m\eta^2 \Gamma^2 \Lambda^2 \mathbb E[f(w_{k}) - f(w_*)] \\
  & \quad - 2m \eta(\gamma \lambda - \eta\Gamma^2\Lambda^2) \mathbb E[f(w_{k+1}) - f(w_*)] .
\end{aligned}
\end{equation*}
Rearranging the above gives
\begin{equation*}
\begin{aligned}
  0 \le & \,\, \mathbb E[f(w_{k}) - f(x_m)]  + 2m\eta^2 \Gamma^2 \Lambda^2 \mathbb E[f(w_{k}) - f(w_*)] \\
  & \quad - 2m \eta(\gamma \lambda - \eta\Gamma^2\Lambda^2) \mathbb E[f(w_{k+1}) - f(w_*)] \\
  \le & \,\, \mathbb E[f(w_{k}) - f(w_*)]  + 2m\eta^2 \Gamma^2 \Lambda^2 \mathbb E[f(w_{k}) - f(w_*)] \\
  & \quad - 2m \eta(\gamma \lambda - \eta\Gamma^2\Lambda^2) \mathbb E[f(w_{k+1}) - f(w_*)] \\
  = & \,\, (1 + 2m\eta^2 \Gamma^2 \Lambda^2) \mathbb E[f(w_{k}) - f(w_*)] \\
  & \quad - 2m \eta(\gamma \lambda - \eta\Gamma^2\Lambda^2) \mathbb E[f(w_{k+1}) - f(w_*)] .
\end{aligned}
\end{equation*}
The second inequality follows from the fact that~$f(w_*) \le f(x_m)$. 
Using the fact that~$\eta < \gamma\lambda/(2\Gamma^2\Lambda^2)$, it follows that
\begin{equation*}
\begin{aligned}
  & \,\, \mathbb E[f(w_{k+1}) - f(w_*)]  \\
  \le & \,\, \frac{1 + 2m\eta^2 \Gamma^2 \Lambda^2}{2m \eta( \gamma \lambda - \eta\Gamma^2\Lambda^2)} \mathbb E[f(w_{k}) - f(w_*)] .
\end{aligned}
\end{equation*}
Since we chose~$m$ and~$\eta$ to satisfy \eqref{eq:eta_m_size}, it follows that the rate~$\alpha$ is less than one. 
This completes the proof.
\end{proof}

\begin{figure*}[t]
        \begin{subfigure}[b]{0.328\textwidth}
                \centering
                \includegraphics[scale=1]{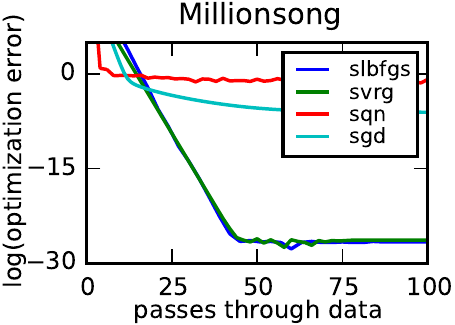}
        \end{subfigure}%
        ~
        \begin{subfigure}[b]{0.328\textwidth}
                \centering
                \includegraphics[scale=1]{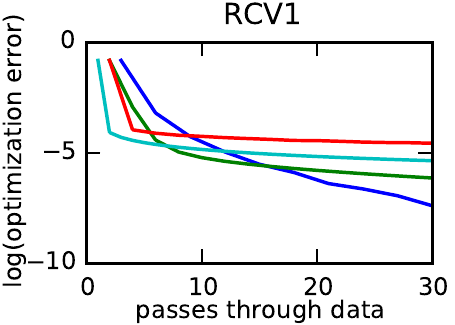}
        \end{subfigure}%
        ~
        \begin{subfigure}[b]{0.328\textwidth}
                \centering
                \includegraphics[scale=1]{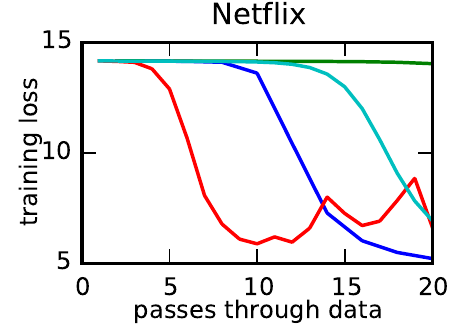}
        \end{subfigure}%

        \caption{
          The left figure plots the log of the optimization error as a function of the number of passes through the data for SLBFGS, SVRG, SQN, and SGD for a ridge regression problem (Millionsong). 
          The middle figure does the same for a support vector machine (RCV1). 
          The right plot shows the training loss as a function of the number of passes through the data for the same algorithms for a matrix completion problem (Netflix). 
        }
        \label{fig:exps}
\end{figure*}

\begin{figure*}[t]
        \begin{subfigure}[b]{0.5\textwidth}
                \centering
                \includegraphics[scale=1]{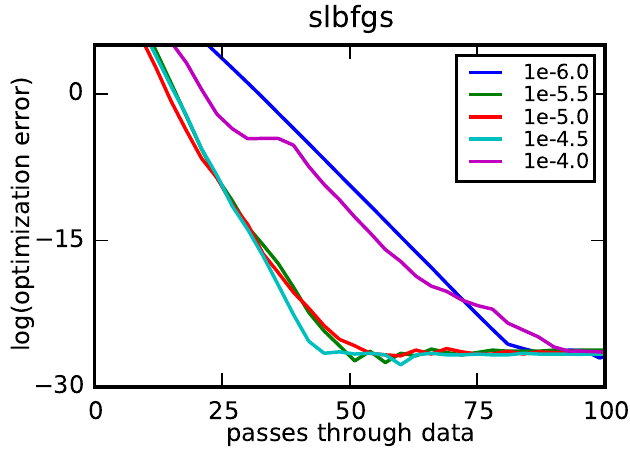}
        \end{subfigure}%
        ~
        \begin{subfigure}[b]{0.5\textwidth}
                \centering
                \includegraphics[scale=1]{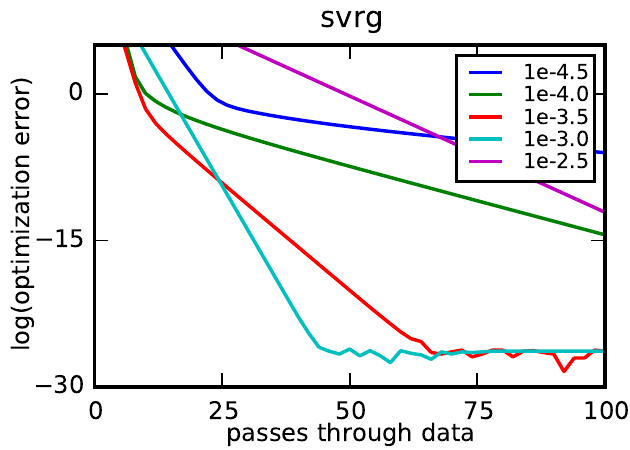}
        \end{subfigure}%

        \begin{subfigure}[b]{0.5\textwidth}
                \centering
                \includegraphics[scale=1]{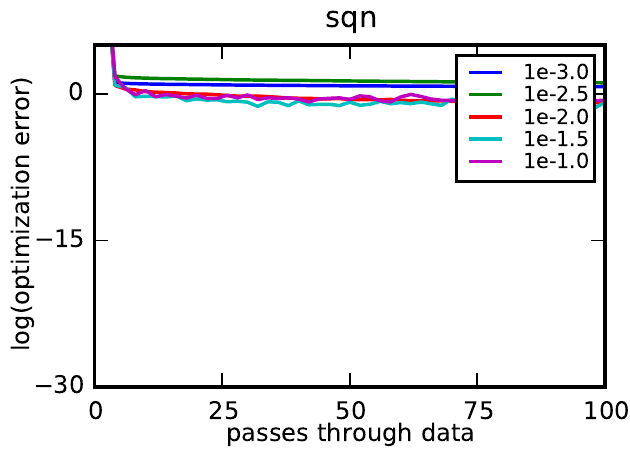}
        \end{subfigure}%
        ~
        \begin{subfigure}[b]{0.5\textwidth}
                \centering
                \includegraphics[scale=1]{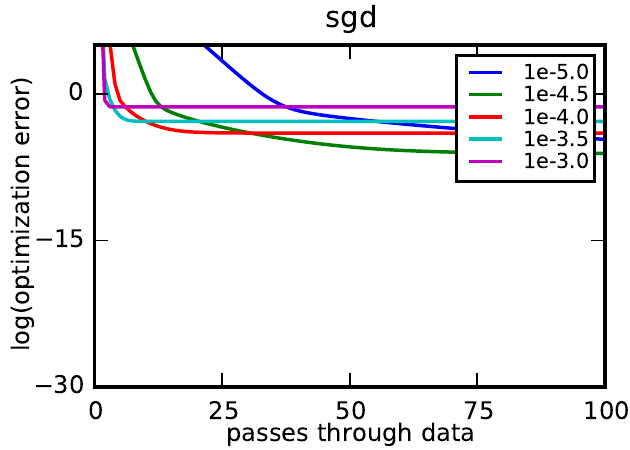}
        \end{subfigure}%
        \caption{These figures show the log of the optimization error for SLBFGS, SVRG, SQN, and SGD on a ridge regression problem (millionsong) for a wide range of step sizes.}
        \label{fig:stepsize_millionsong}
\end{figure*}

\begin{figure*}[t]
        \begin{subfigure}[b]{0.5\textwidth}
                \centering
                \includegraphics[scale=1]{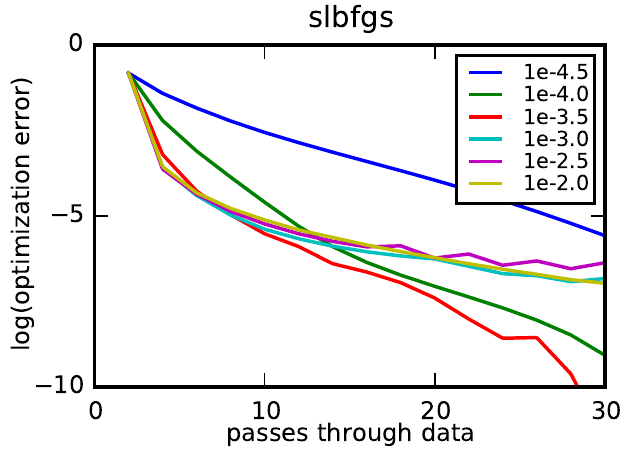}
        \end{subfigure}%
        ~
        \begin{subfigure}[b]{0.5\textwidth}
                \centering
                \includegraphics[scale=1]{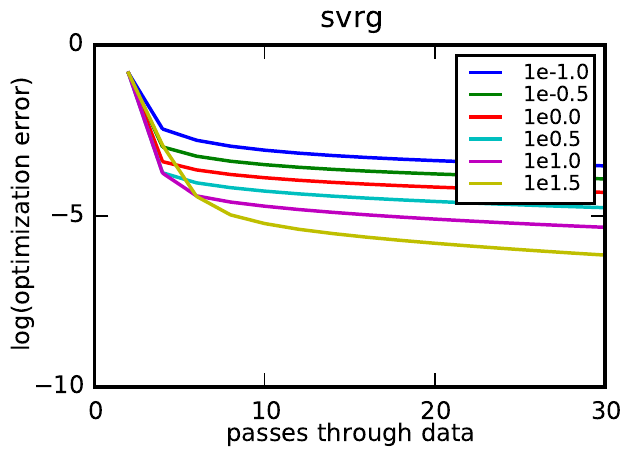}
        \end{subfigure}%

        \begin{subfigure}[b]{0.5\textwidth}
                \centering
                \includegraphics[scale=1]{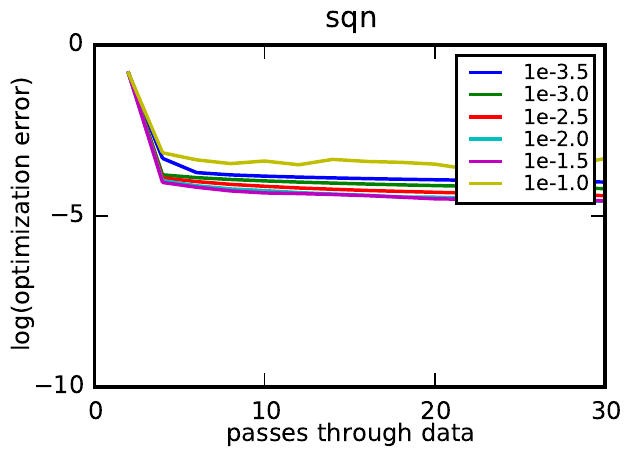}
        \end{subfigure}%
        ~
        \begin{subfigure}[b]{0.5\textwidth}
                \centering
                \includegraphics[scale=1]{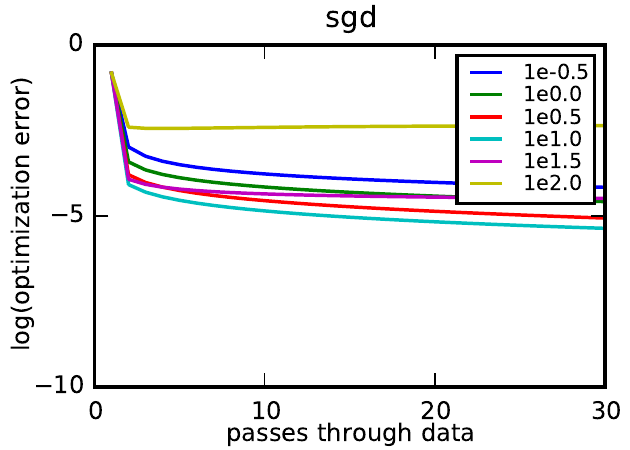}
        \end{subfigure}%
        \caption{These figures show the log of the optimization error for SLBFGS, SVRG, SQN, and SGD on a support vector machine (RCV1) for a wide range of step sizes.}
        \label{fig:stepsize_rcv1}
\end{figure*}

\section{Related Work}

There is a large body of work that attempts to improve on stochastic gradient descent by reducing variance. 
\citet{shalev2013stochastic} propose stochastic dual coordinate ascent (SDCA). 
\citet{roux2012stochastic} propose the stochastic average gradient method (SAG). 
\citet{johnson2013accelerating} propose the stochastic variance reduced gradient (SVRG). 
\citet{wang2013variance} develop an approach based on the construction of control variates. 
More recently, \citet{frostig2015competing} devise an online version of SVRG that uses streaming estimates of the gradient to perform variance reduction. 

Similarly, a number of stochastic quasi-Newton methods have been proposed. 
\citet{bordes2009sgd} propose a variant of stochastic gradient descent that makes use of second order information. 
\citet{mokhtari2014global} analyze the straightforward application of L-BFGS in the stochastic setting and prove a~$O(1/k)$ convergence rate in the strongly-convex setting. 
\citet{byrd2014stochastic} propose a modified version of L-BFGS in the stochastic setting and prove a~$O(1/k)$ convergence rate in the strongly-convex setting. 
\citet{sohl2014fast} propose a stochastic quasi-Newton method for minimizing sums of functions by maintaining a separate approximation of the inverse Hessian for each function in the sum. 
\citet{schraudolph2007stochastic} develop a stochastic version of L-BFGS for the online convex optimization setting. 
\citet{wang2014stochastic} prove the convergence of various stochastic quasi-Newton methods in the nonconvex setting. 
Our work differs from the preceding in that we guarantee a linear rate of convergence. 

\citet{lucchi2015variance} independently propose a variance-reduction procedure to speed up stochastic quasi-Newton methods and to achieve a linear rate of convergence. 
Their approach to updating the inverse-Hessian approximation is similar to that of L-BFGS, whereas our method leverages Hessian-vector products to stabilize the approximation. 

\section{Experimental Results} \label{sec:experiments}

To probe our theoretical results, we compare \algoref{alg:slbfgs} (SLBFGS) to the stochastic variance-reduced gradient method (SVRG) \citep{johnson2013accelerating}, the stochastic quasi-Newton method (SQN) \citep{byrd2014stochastic}, and stochastic gradient descent (SGD). 
We evaluate these algorithms on several popular machine learning models, including ridge regression, support vector machines, and matrix completion.
Our experiments show the effeciveness of the algorithm on real-world problems that are not neccessarily (strongly) convex.

Because SLBFGS and SVRG require computations of the full gradient, each epoch requires an additional pass through the data. 
Additionally, SLBFGS and SQN require Hessian-vector-product calculations, each of which is about as expensive as a gradient calculation \citet{pearlmutter1994fast}. 
The number of Hessian-vector-product computations per epoch introduced by this is~$(b_H N)/(b L)$, which in our experiments is either~$N$ or~$2N$. 
To incorporate these additional costs, our plots show error with respect to the number of passes through the data (that is, the number of gradient or Hessian-vector-product computations divided by~$N$). 
For this reason, the first iterations of SLBFGS, SVRG, SQN, and SGD all begin at different times, with SGD appearing first and SLBFGS appearing last.

For all experiments, we set the batch size~$b$ to either~$20$ or~$100$, we set the Hessian batch size~$b_H$ to~$10b$ or~$20b$, we set the Hessian update interval~$L$ to~$10$, we set the memory size~$M$ to~$10$, and we set the number of stochastic updates~$m$ to~$N/b$. 
We optimize the learning rate via grid search. 
SLBFGS and SVRG use a constant step size. 
For SQN and SGD, we try three different step-size schemes: constant,~$1/\sqrt{t}$, and~$1/t$, and we report the best one. 
All experiments are initialized with a vector of zeros, except for the matrix completion problem, where in order to break symmetry, we initialize the experiments with a vector of standard normal random variables scaled by~$10^{-5}$. 

First, we performed ridge regression on the millionsong dataset \citep{bertin-mahieux2011} consisting of approximately~$4.6 \times 10^5$ data points. 
We set the regularization parameter~$\lambda=10^{-3}$. 
In this experiment, both SLBFGS and SVRG rapidly solve the problem to high levels of precision. 
Second, we trained a support vector machine on RCV1 \citep{lewis2004rcv1}, with approximately~$7.8 \times 10^6$ data points. 
We set the regularization parameter to~$\lambda=0$. 
In this experiment, SGD and SQN make more progress initially as expected, but SLBFGS finds a better optimum. 
Third, we solve a nonconvex matrix completion problem on the Netflix Prize dataset, as formulated in \citet{recht2013parallel}, with approximately~$10^{8}$ data points. 
We set the regularization parameter to~$\lambda=10^{-4}$. 
The poor performance of SVRG and SGD on this problem may be accounted for by the fact that the algorithms are initialized near the vector of all zeros, which is a stationary point (though not the optimum). 
Presumably the use of curvature information helps SLBFGS and SQN escape the neighborhood of the all zeros vector faster than SVRG and SGD. 

\figref{fig:exps} plots a comparison of these methods on the three problems. 
For the convex problems, we plot the logarithm of the optimization error with respect to a precomputed reference solution. 
For the nonconvex problem, we simply plot the objective value as the global optimum is not necessarily known.

\subsection{Robustness to Choice of Step Size}

In this section, we illustrate that SLBFGS performs well on convex problems for a large range of step sizes. 
The windows in which SVRG, SQN, and SGD perform well are much narrower. 
In \figref{fig:stepsize_millionsong}, we plot the performance of SLBFGS, SVRG, SQN, and SGD for ridge regression on the millionsong dataset for step sizes varying over a couple orders of magnitude. 
In \figref{fig:stepsize_rcv1}, we show a similar plot for a support vector machine on the RCV1 dataset. 
In both cases, SLBFGS performs well, solving the problem to a high degree of precision over a large range of step sizes, whereas the performance of SVRG, SQN, and SGD degrade much more rapidly with poor step-size choices.

\section{Proofs of Preliminaries}

\subsection{Proof of \lemref{lem:trace_and_det_bounds}} \label{sec:proof_of_lem_trace_and_det_bounds}
The analysis below closely follows many other analyses of the inverse Hessian approximation used in L-BFGS \citep{nocedal2006numerical,byrd2014stochastic,mokhtari2014global,mokhtari2014res}, and we include it for completeness. 

Note that~$s_j^{\top}y_j=s_j \nabla^2 f_{\mathcal T_j}(u_j) s_j$, it follows from \assref{ass:hessian_bounds} that
\begin{equation} \label{eq:bound_s_sy}
  \lambda \|s_j\|^2 \le s_j^{\top}y_j \le \Lambda \|s_j\|^2 .
\end{equation}
Similarly, letting~$z_j=(\nabla^2 f_{\mathcal T_j}(u_j))^{1/2} s_j$ and noting that
\begin{equation*}
  \frac{\|y_j\|^2}{s_j^{\top}y_j} = \frac{z_j^{\top} \nabla^2 f_{\mathcal T_j}(u_j) z_j}{z_j^{\top}z_j} ,
\end{equation*}
\assref{ass:hessian_bounds} again implies that
\begin{equation} \label{eq:bound_y_sy}
  \lambda \le \frac{\|y_j\|^2}{s_j^{\top}y_j} \le \Lambda .
\end{equation}

Note that using the Sherman-Morrison-Woodbury formula, we can equivalently write \eqref{eq:inv_hess_update} in terms of the Hessian approximation~$B_r=H_r^{-1}$ as
\begin{equation} \label{eq:hess_update}
  B_r^{(j)} = B_r^{(j-1)} - \frac{B_r^{(j-1)} s_j s_j^{\top} B_r^{(j-1)}}{s_j^{\top} B_r^{(j-1)} s_j} + \frac{y_j y_j^{\top}}{y_j^{\top}s_j} .
\end{equation}
We will begin by bounding the eigenvalues of~$B_r$. 
We will do this indirectly by bounding the trace and determinant of~$B_r$. 
We have
\begin{align*}
  \tr(B_r^{(j)}) & = \tr(B_r^{(j-1)}) - \frac{\tr(B_r^{(j-1)} s_j s_j^{\top} B_r^{(j-1)})}{s_j^{\top} B_r^{(j-1)} s_j} + \frac{\tr(y_j y_j^{\top})}{y_j^{\top}s_j} \\
  & = \tr(B_r^{(j-1)}) - \frac{\|B_r^{(j-1)} s_j\|^2}{s_j^{\top} B_r^{(j-1)} s_j} + \frac{\|y_j\|^2}{y_j^{\top}s_j} \\
  & \le \tr(B_r^{(j-1)}) + \frac{\|y_j\|^2}{y_j^{\top}s_j} \\
  & \le \tr(B_r^{(j-1)}) + \Lambda .
\end{align*}
The first equality follows from the linearity of the trace operator. 
The second equality follows from the fact that $\tr(AB)=\tr(BA)$. 
The fourth relation follows from \eqref{eq:bound_y_sy}. 
Since
\begin{equation*}
  \tr(B_r^{(0)}) = d \frac{\|y_r\|^2}{s_r^{\top}y_r} \le d\Lambda ,
\end{equation*}
it follows inductively that
\begin{equation*}
  \tr(B_k) \le (d+M)\Lambda .
\end{equation*}

Now to bound the determinant, we write
\begin{align*}
  \det(B_r^{(j)}) & = \det(B_r^{(j-1)}) \\
  & \quad\, \det\left( I - \frac{s_j s_j^{\top} B_r^{(j-1)}}{s_j^{\top} B_r^{(j-1)} s_j} + \frac{(B_r^{(j-1)})^{-1}y_j y_j^{\top}}{y_j^{\top}s_j} \right) \\
  & = \det(B_r^{(j-1)}) \frac{y_j^{\top}s_j}{s_j^{\top} B_r^{(j-1)} s_j} \\
  & = \det(B_r^{(j-1)}) \frac{y_j^{\top}s_j}{\|s_j\|^2} \frac{\|s_j\|^2}{s_j^{\top} B_r^{(j-1)} s_j} \\
  & \ge \det(B_r^{(j-1)}) \frac{\lambda}{\lambda_{\max}(B_r^{(j-1)})} \\
  & \ge \det(B_r^{(j-1)}) \frac{\lambda}{\tr(B_r^{(j-1)})} \\
  & \ge \det(B_r^{(j-1)}) \frac{\lambda}{(d+M)\Lambda} .
\end{align*}
The first equality uses~$\det(AB)=\det(A)\det(B)$. 
The second equality follows from the identity
\begin{align} \label{eq:det_rank_two_update}
  & \,\, \det(I + u_1v_1^{\top} + u_2v_2^{\top}) \\
  = & \,\, (1 + u_1^{\top}v_1)(1 + u_2^{\top}v_2) - (u_1^{\top}v_2)(v_1^{\top}u_2) \nonumber
\end{align}
by setting $u_1 = -s_j$, $v_1 = (B_r^{(j-1)}s_j)/(s_j^{\top}B_r^{(j-1)}s_j)$, $u_2 = (B_r^{(j-1)})^{-1} y_j$, and $v_2 = y_j/(y_j^{\top}s_j)$.
See \citet[Lemma~7.6]{dennis1977quasi} for a proof, or simply note that \eqref{eq:det_rank_two_update} follows from two applications of the identity~$\det(A+uv^{\top})=(1+v^{\top}A^{-1}u)\det(A)$ when~$I+u_1v_1^{\top}$ is invertible and by continuity when it isn't. 
The third equality follows by multiplying the numerator and denominator by~$\|s_j\|^2$. 
The fourth relation follows from \eqref{eq:bound_s_sy} and from the fact that $s_j^{\top}B_r^{(j-1)}s_j \le \lambda_{\max}(B_r^{(j-1)})\|s_j\|^2$. 
The fifth relation uses the fact that the largest eigenvalue of a positive definite matrix is bounded by its trace. 
The sixth relation uses the previous bound on~$\tr(B_r^{(j-1)})$. 
Since


\begin{equation*}
  \det(B_r^{(0)}) = \left( \frac{\|y_r\|^2}{s_r^{\top}y_r} \right)^d \ge \lambda^d ,
\end{equation*}
it follows inductively that
\begin{equation*}
  \det(B_r) \ge \frac{\lambda^{d+M}}{((d+M)\Lambda)^M} .
\end{equation*}

\subsection{Proof of \lemref{lem:hess_approx_bounds}}  \label{sec:proof_of_lem_hess_approx_bounds}
Using \lemref{lem:trace_and_det_bounds} as well as the fact that~$H_r$ is positive definite, we have
\begin{equation*}
  \lambda_{\max}(B_r) \le \tr(B_r) \le (d+M)\Lambda .
\end{equation*}
and
\begin{equation*}
  \lambda_{\min}(B_r) \ge \frac{\det(B_r)}{\lambda_{\max}(B_r)^{d-1}} \ge \frac{\lambda^{d+M}}{((d+M)\Lambda)^{d+M-1}} .
\end{equation*}
Since we defined~$B_r=H_r^{-1}$, it follows that
\begin{equation*}
  \frac{1}{(d+M)\Lambda}I \preceq H_r \preceq \frac{((d+M)\Lambda)^{d+M-1}}{\lambda^{d+M}} I  .
\end{equation*}

\subsection{Proof of \lemref{lem:bound_grad_var}} \label{sec:proof_of_bound_grad_var}
  Define the function~$g_{\mathcal S}(w)=f_{\mathcal S}(w) - f_{\mathcal S}(w_*) - \nabla f_{\mathcal S}(w_*)^{\top}(w-w_*)$ to get the linearization of~$f_{\mathcal S}$ around the optimum~$w_*$, and note that~$g_{\mathcal S}$ is minimized at~$w_*$. 
  It follows that for any~$w$, we have
  \begin{equation*}
    0 = g_{\mathcal S}(w_*) \le g_{\mathcal S}\left(w - \frac{1}{\Lambda} \nabla g_{\mathcal S}(w)\right) \le g_{\mathcal S}(w) - \frac{1}{2\Lambda} \|\nabla g_{\mathcal S}\|^2 .
  \end{equation*}
  Rearranging, we have
  \begin{align*}
    & \,\, \|\nabla f_{\mathcal S}(w) - \nabla f_{\mathcal S}(w_*)\|^2 \\
    \le & \,\, 2\Lambda(f_{\mathcal S}(w) - f_{\mathcal S}(w_*) - \nabla f_{\mathcal S}(w_*)^{\top}(w-w_*)) .
  \end{align*}
  Averaging over all possible minibatches~$\mathcal S \subseteq \{1,\ldots,N\}$ of cardinality~$b$ and using the fact that~$\nabla f(w_*)=0$, we see that
  \begin{align} \label{eq:grad_norm_bound}
    & \,\, {N \choose b}^{-1} \sum_{|\mathcal S|=b} \|\nabla f_{\mathcal S}(w) - \nabla f_{\mathcal S}(w_*)\|^2 \\
    \le & \,\, 2\Lambda(f(w) - f(w_*)) .\nonumber 
  \end{align}
  Now, let~$\mu_{k}=\nabla f(w_{k})$ and~$v_{t}=\nabla f_{\mathcal S}(x_{t}) - \nabla f_{\mathcal S}(w_{k}) + \mu_{k}$. 
  Conditioning on~$\mathcal F_{k,t}$ and taking an expectation with respect to~$\mathcal S$, we find
  \begin{align} \label{eq:bound_grad_var}
    \mathbb E_{k,t}[\|v_{t}\|^2] \le & \,\, 2 \mathbb E_{k,t}[\|\nabla f_{\mathcal S}(x_{t}) - \nabla f_{\mathcal S}(w_*)\|^2]  \\
    & \quad + 2 \mathbb E_{k,t}[\|\nabla f_{\mathcal S}(w_{k}) - \nabla f_{\mathcal S}(w_*) - \mu_{k}\|^2] \nonumber \\
    \le & \,\, 2 \mathbb E_{k,t}[\|\nabla f_{\mathcal S}(x_{t}) - \nabla f_{\mathcal S}(w_*)\|^2]  \nonumber \\
    & \quad + 2 \mathbb E_{k,t}[\|\nabla f_{\mathcal S}(w_{k}) - \nabla f_{\mathcal S}(w_*)\|^2]  \nonumber \\
    \le & \,\, 4\Lambda(f(x_{t})-f(w_*) + f(w_{k})-f(w_*)) . \nonumber 
  \end{align}
  The first inequality uses the fact that~$\|a+b\|^2 \le 2\|a\|^2 + 2\|b\|^2$. 
  The second inequality follows by noting that~$\mu_{k} = \mathbb E_{k,t}[\nabla f_{\mathcal S}(w_{k}) - \nabla f_{\mathcal S}(w_*)]$ and that~$\mathbb E[\|\xi-\mathbb E[\xi]\|^2] \le \mathbb E[\|\xi\|^2]$ for any random variable~$\xi$. 
The third inequality follows from \eqref{eq:grad_norm_bound}.

\section{Discussion}
This paper introduces a stochastic version of L-BFGS and proves a linear rate of convergence in the strongly convex case. 
\thref{th:linear_convergence} captures the qualitatively linear rate of convergence of SLBFGS, which is reflected in our experimental results. 
We expect SLBFGS to outperform other stochastic first-order methods in poorly conditioned settings where curvature information is valuable as well in settings where we wish to solve the optimization problem to high precision. 

There are a number of interesting points to address in future work. 
The proof of \thref{th:linear_convergence} and many similar proofs used to analyze quasi-Newton methods result in constants that scale poorly with the problem size.
At a deeper level, the point of studying quasi-Newton methods is to devise algorithms that lie somewhere along the spectrum from gradient descent to Newton's method, reaping the computational benefits of gradient descent and the rapid convergence of Newton's method. 
Many of the proofs in the literature, including the proof of \thref{th:linear_convergence}, bound the extent to which the quasi-Newton method deviates from gradient descent by bounding the extent to which the inverse Hessian approximation deviates from the identity matrix. 
Those bounds are then used to show that the quasi-Newton method does not perform too much worse than gradient descent.
A future avenue of research is to study if stochastic quasi-Newton methods can be designed that provably exhibit superlinear convergence as has been done in the non-stochastic case.


{
\bibliographystyle{abbrvnat}
\bibliography{refs}
}

\end{document}